\documentclass[a4paper,10pt]{article}
\usepackage[utf8]{inputenc}
\usepackage[T1]{fontenc}
\usepackage[english]{babel}
\usepackage{amsmath}
\usepackage{amsthm}
\usepackage{amsfonts}
\usepackage{amssymb}
\usepackage{listings}
\usepackage{graphicx}
\usepackage{mathrsfs}
\usepackage{hyperref}
\usepackage{algpseudocode}
\usepackage{algorithmicx}
\usepackage{algorithm}
\usepackage{subfig}
\usepackage{stmaryrd}
\usepackage{bm}
\usepackage{tikz}
\usepackage{braket}
\usepackage{wasysym}
\usepackage{wrapfig}
\usepackage[framemethod=TikZ]{mdframed}
\usepackage{xcolor}
\usepackage{pstricks}
\usepackage{faktor}

\newcommand{\mc}{\mathscr}
\newcommand{\f}{\mathbb}

\newcommand{\cu}{\subseteq}

\newcommand{\serie}[1]{\{#1_{n}\}_n}
\newcommand{\GLT}{\sim_{GLT}}
\newcommand{\dacs}[2]{d_{acs}\left(#1,#2\right)}
\newcommand{\acs}{\xrightarrow{a.c.s.}}
\newcommand{\B}{\{B_{n,m}\}_{n,m}}
\newcommand{\ve}{\varepsilon}

\newtheorem{teo}{Theorem}[section]
\newtheorem{lemma}{Lemma}[section]
\newtheorem{corollario}{Corollary}[section]

\DeclareMathOperator{\ess}{ess}
\DeclareMathOperator{\rk}{rk}

\title{Equivalence between GLT sequences and measurable functions}
\author{Giovanni Barbarino}

\begin{document}

\maketitle

\begin{abstract}
The theory of Generalized Locally Toeplitz (GLT) sequences of matrices has been developed in order to study the asymptotic behaviour of particular spectral distributions when the dimension of the matrices tends to infinity. A key concepts in this theory are the notion of Approximating Classes of Sequences (a.c.s.), and spectral symbols, that lead to define a metric structure on the space of matrix sequences, and provide a link with the measurable functions.
In this document we prove additional results regarding theoretical aspects, such as the completeness of the matrix sequences space with respect to the metric a.c.s., and the identification of the space of GLT sequences with the space of measurable functions.
\end{abstract}

\section{Introduction}

When dealing with the discretization of differential equations, we often have to solve sequences of linear equations in the form $A_nx=b_n$, where $A_n\in \f C^{n\times n}$. The dimension of the matrices is determined by the degree of refinement of the mesh used in the Finite Difference methods, or the dimension of the subspace used in Finite Element methods. Solving high-dimensional linear equations is fundamental to get accurate solutions, but the rate of convergence of the solvers (Conjugate Gradient, Preconditioned Krylov methods, Multigrid techniques etc.) depends on the spectra of the matrices, so the knowledge of the asymptotic distribution of the sequence $\serie A$ is a strong tool we can use to choose or to design the best solver and method of discretization (see \cite{BS},\cite{GSM} and references therein).

These are some of the reasons that lead to the study of \textit{spectral symbols} of matrix sequences. We recall that a spectral symbol associated with a sequence $\serie A$ is a measurable functions $k:D\cu \f R^n\to \f C$, where $D$ is measurable set with finite non-zero Lebesgue measure, satisfying 
\[
\lim_{n\to\infty} \frac{1}{n} \sum_{i=1}^{n} F(\sigma_i(A_n)) = \frac{1}{|D|}\int_D F(|k(x)|) dx
\]
for every $F:\f R\to \f C$ continuous function with compact support.
Here $|D|$ is the Lebesgue measure of $D$, and 
\[\sigma_1(A_n)\ge \sigma_2(A_n)\ge\dots\ge \sigma_n(A_n)\]
are the singular values in non-increasing order. In this case, we will say that $\serie A$ has spectral symbol $k$ and we will write \[\serie A\sim_\sigma k.\]
The function $k$ thus becomes an asymptotic singular values distribution, but in general it is not uniquely determined.

The space of matrix sequences that admit a spectral symbol on a fixed domain $D$ has been shown to be closed with respect to a notion of convergence called the Approximating Classes of Sequences (a.c.s.). This notion and this result are due to Serra \cite{ACS}, but were
actually inspired by Tilli’s pioneering paper on LT sequences \cite{Tilli}. Given a sequence of matrix sequences $\B$, it is said to be a.c.s. convergent to $\serie A$ if there exist a sequence $\{N_{n,m}\}_{n,m}$ of "small norm" matrices and a sequence  $\{R_{n,m}\}_{n,m}$ of "small rank" matrices such that for every $m$ there exists $n_m$ with
\[
A_n = B_{n,m}  + N_{n,m} + R_{n,m}, \qquad \|N_{n,m}\|\le \omega(m), \qquad \rk(R_{n,m})\le nc(m)
\]
for every $n>n_m$, and
\[
\omega(m)\xrightarrow{m\to \infty} 0,\qquad c(m)\xrightarrow{m\to \infty} 0.
\]
In this case, we will use the notation $\B\acs \serie A$. The result of closeness tells us that if $\B \sim_\sigma k_m$ for every $m$, $k_m\to k$ in measure, and $\B\acs \serie A$, then $\serie A\sim_\sigma k$. This result is central in the theory since it lets us compute the spectral symbol of a.c.s. limits, and it is useful when we can find simple sequences that converge to the wanted $\serie A$.

In this context, it has been observed that the sequences $\serie A$ arising from differential equations can often be obtained through limit a.c.s. of sum and products of special diagonal and Toeplitz sequences, for which we can easily deduce the spectral symbol. This justifies the interest on the space of Generalized Locally Toeplitz (GLT) sequences, that contains both the before-mentioned class of sequences, gains the structure of $\f C$-algebra, and is closed with respect to the a.c.s. convergence. Here we will report only few properties of this space, but for a detailed presentation of the GLT sequences and their applications refer to \cite{BS},\cite{GLT},\cite{Tilli},\cite{GSM} and references therein.
The space of GLT is built so that for every GLT matrix sequence $\serie A$ we choose one of its spectral symbol $k$ in particular, up to identification of the function almost everywhere (a.e.), and we denote it as $\serie A\GLT k$. 
This choice ensure us that the GLT space 
\[
\mc G = \left\{ (\serie A,k)\in \mc E\times \mc M_D : \serie A\GLT k     \right\}
\]
is a $\f C-$algebra, where
\[
\mc M_D = \{k:D\to \f C, \medspace k \text{ measurable }\}/ \sim
\]
\[
k\sim k' \iff k=k' \text{ a.e.}
\]
and
\[
\mc E := \{\serie{A} : A_n\in\f C^{n\times n} \}.
\]
We will see more properties of $\mc G$ in Section 3.

We know that $\mc M_D$ is endowed with a complete metric that induces the convergence in measure of measurable functions, and we can also endow the set of matrix sequences $\mc E$ with a pseudometric that induces the a.c.s. convergence mentioned before. In section 2, we will define the pseudometric and introduce an easy way to compute it.
In the works \cite{GS} and \cite{GStr}, it is shown that $\mc G$ is a closed metric subspace of $\mc E\times \mc M_D$, and  the connection between the distances on $\mc E$ and $\mc M_D$ is emphasized.

The paper is organized as follows. In Section 2 we recall the definition of the pseudometric $d_{acs}$ on $\mc E$ and show it is actually complete. Given $d_m$ a particular distance in $\mc M_D$ that induces the convergence in measure, we also prove that if a sequence $\serie A$ has spectral symbol $k$, then $d_m(k,0) = \dacs{\serie A}{\serie{0}}$, where $\serie 0$ is the sequence of zero matrices.
Some fundamental properties of GLT sequences are reported in Section 3, where we use the previously obtained results to prove that,  up to a.c.s. equivalence, the set of GLT sequences is actually isomorphic and isometric to the space of measurable function $\mc M_D$. Eventually we prove that the GLT algebra is already a maximal group in the space of the sequences that admit a spectral symbol, and cannot be expanded anymore.


\section{Complete Pseudometric}

Given a matrix $A\in\f C^{n\times n}$, we can define the function
\[
p(A):= \min_{i=1,\dots,n}\left\{ \frac{i-1}{n} + \sigma_i(A) \right\}
\]
that respects the triangular inequality. In fact, given $A,B$ matrices with the same dimension, we have
\[
p(A+B)\le p(A) + p(B).
\] 
Given now a sequence $\serie A\in\mc E$, we can denote
\[
\rho\left(\serie A\right):= \limsup_{n\to \infty} p(A_n) 
\]
that lets us to introduce a pseudometric $d_{acs}$ on $\mc E$
\[
\dacs{\serie{A}}{\serie{B}} = \rho\left(\{A_n-B_n\}_n\right).
\]
It has been proved (\cite{Garoni}) that this distance induces the a.c.s. convergence already introduced. In other words,
\[
\dacs{\serie A}{\B} \xrightarrow{m\to \infty} 0 \iff  \B\acs \serie A.
\]

%

The next two sections shows that this pseudometric is complete, and that it is strictly linked to a complete metric of the measurable functions.

\subsection{Completeness}

In this section, we prove the completeness of the space $\mc E$ endowed with the pseudometric $d_{acs}$.

\begin{teo}\label{comp}
The set $\mc E$ is complete with the pseudometric  $d_{acs}$.
\end{teo}
\begin{proof}
Let $\B$ be a Cauchy sequence for the metric. The convergence of the sequence is equivalent to the convergence of any subsequence, with the same limit, so we can always extract a subsequence such that for every couple of indeces  $s,t$ 
\[
\dacs{\{B_{n,s}\}_n}{\{B_{n,t}\}_n} \le 2^{-\min\{s,t\}}
\] 
or, equivalently
\[
\limsup_{n\to \infty} p(B_{n,s} - B_{n,t})\le 2^{-\min\{s,t\}}.
\]
If we consider the indices   $m$ and $m+1$, we obtain
\[
\limsup_{n\to \infty} p(B_{n,m} - B_{n,m+1})\le 2^{-m}.
\]
We know that, given an $\ve >0$, the argument of the limsup is eventually less then $2^{-m}+\ve$. Thus we choose $\ve = 2^{-m}$ and find a strictly increasing sequence of indices $N_m$ such that 
\[
 p(B_{n,m} - B_{n,m+1})\le 2^{-m+1} \qquad \forall n\ge  N_m.
\]
We can now build the sequence $\serie{A}$ that will be our limit guess.
\[
A_n := B_{n,m} \qquad \text{ whenever } \qquad N_{m+1}>n\ge N_m
\]
If $N_{M+1}>n\ge N_M$ with $M\ge m$, then we can estimate the distance between $A_n$ and $B_{n,m}$ as
\[
p(B_{n,m}-A_n) = p(B_{n,m}-B_{n,M}) \le \sum_{k=m}^{M-1} p(B_{n,k}-B_{n,k+1}),
\]
but $n\ge N_M> N_k$ for all indices $k$ in the summation, so
\[
\sum_{k=m}^{M-1} p(B_{n,k}-B_{n,k+1})\le \sum_{k=m}^{M-1} 2^{-k+1}\le 2\cdot 2^{-m} \sum_{k=0}^{M-m-1} 2^{-k}	\le 4\cdot 2^{-m}.
\]
The latter bound does not depend on $M$ anymore, so we conclude that
\[
p(B_{n,m}-A_n)\le 4\cdot 2^{-m} \qquad \forall m\quad \forall n\ge N_m
\]
\[
\implies \dacs{\{B_{n,m}\}_{n}}{\serie A} =  \limsup_{n\to \infty} p(B_{n,m}-A_n)\le 4\cdot 2^{-m} \xrightarrow{m\to \infty} 0.
\]
\end{proof}

Theorem \ref{comp} tells us that if we have a sequence $\B$ we can claim its convergence just by checking if it is a Cauchy sequence. Moreover, we can actually build the limit sequence by following the proof of the theorem.

\subsection{Equivalence of Distances}

On the space $\mc M_D$ we can define the function
\[
p_m(f) := \inf \left\{ \frac{|E^C|}{|D|} + \ess\sup_E |f|  \right\},
\]
where the inf is taken on all the measurable sets $E\cu D$ and $E^C$ stands for the complementary of $E$. This function induces the complete distance 
\[
d_m(f,g) = p_m(f-g)
\]
and in this section we prove that if $\serie A\sim_\sigma g$ then 
\[
d_m(g,0) = p_m(g) = \rho(\serie A) = \dacs{\serie A}{\serie{0}}
\]
where $\serie 0$ is the sequence of zero matrices and growing size $n$.
\begin{lemma}
Given $\serie A\in\mc E$ and $g\in \mc M_D$, we have
\[
\serie{A}\sim_\sigma g,\quad  p_m(g) = L\implies \rho(\serie{A})\le L.
\]
\end{lemma}
\begin{proof}

We know that
\[
 p_m(g)= \inf_{E\cu D} \left\{ \frac{|E^C|}{|D|} + \ess\sup_E |g|  \right\}= L,
\]
where the sets $E$  are Lebesgue measurable.
By definition of inferior, if we set $\ve>0$, we can always find  $F$ such that 
\[
 \frac{|F^C|}{|D|}+ \ess\sup_F |g| \le L + \ve.
\]
From now on, let us call $M=\ess\sup_F |g|$.
Since $\serie A \sim_\sigma g$, we know that
\[
\lim_{n\to\infty} \frac{1}{n} \sum_{i=1}^{n} F(\sigma_i(A_n)) = \frac{1}{|D|}\int_D F(|g(x)|) dx
\]
for every function $F:\f R\to \f C$ continuous with compact support. Given such an $F$ real valued and with  $\chi_{[-\ve,M+\ve]}\ge F\ge \chi_{[0,M]}$, we obtain
\[
\lim_{n\to\infty} \frac{1}{n} \sum_{i=1}^{n} F(\sigma_i(A_n)) \le \liminf_{n\to\infty} \frac{1}{n} \left|\left\{ i:\sigma_i(A_n) \le M+\ve  \right\}\right|
\]
and
\[
\int_D F(|g(x)|) dx \ge  |\{x:|g(x)|\le M\}| \ge |F|.
\]
Therefore
\[
\liminf_{n\to\infty} \frac{1}{n} \left|\left\{ i:\sigma_i(A_n) \le M+\ve  \right\}\right|\ge \frac{|F|}{|D|},
\]
\[
\limsup_{n\to\infty} \frac{1}{n} \left|\left\{ i:\sigma_i(A_n) > M+\ve  \right\}\right|\le \frac{|F^C|}{|D|}\le L + \ve - M.
\]
The argument of limsup will eventually be less then $L + 2\ve - M$, so there exists $N>0$ such that 
\[
\frac{1}{n} \left|\left\{ i:\sigma_i(A_n) > M+\ve  \right\}\right|\le L + 2\ve - M
\quad \forall n>N,
\]
\[
\frac{1}{n} \left|\left\{ i:\sigma_i(A_n) > M+\ve  \right\}\right| + M+\ve\le L + 3\ve 
\quad \forall n>N.
\]
This concludes the proof since
\[
\rho(\serie{A}) = \limsup_{n\to\infty}\min_{i=1,\dots,n}\left\{\frac{i-1}{n} + \sigma_i(A_n)\right\}\]
and for every $n>N$ we can choose the greatest $\sigma_i(A_n)$ less or equal than $M+\ve$. Consequently
\[
\min_{i=1,\dots,n}\left\{\frac{i-1}{n} + \sigma_i(A_n)\right\}\le \frac{1}{n} \left|\left\{ i:\sigma_i(A_n) > M+\ve  \right\}\right| + M+\ve \le L+ 3\ve
\]
and hence
\[
\rho(\serie{A})\le \limsup_{n\to\infty} \{ L + 3\ve \} = L+3\ve
\]
for every $\ve>0$, so that the proof in concluded since $\rho(\serie A)\le C$.
\end{proof}

The converse statement has a similar proof

\begin{lemma}
Given $\serie A\in\mc E$ and $g\in \mc M_D$, we have
\[
\serie{A}\sim_\sigma g,\quad  \rho(\serie{A}) = L\implies p_m(g) \le L.
\]
\end{lemma}
\begin{proof}
\[
\rho(\serie{A}) = \limsup_{n\to\infty}\min_{i=1,\dots,n}\left\{\frac{i-1}{n} + \sigma_i(A_n)\right\}= L
\]
so we can set $\ve >0$ and find $N>0$ such that
\[
\min_{i=1,\dots,n}\left\{\frac{i-1}{n} + \sigma_i(A_n)\right\}\le L + \ve\qquad \forall n>N
\]
and if we call $j(n)$ the index $i$ that realizes the minimum, then 
\[
0	\le \frac{j(n)-1}{n} + \sigma_{j(n)}(A_n)\le L+\ve \qquad \forall n>N.
\]
The sequence $\{\sigma_{j(n)}(A_n)\}_n$ is bounded, so we can find a converging subsequence such that
\[
 \{\sigma_{j(n_k)}(A_{n_k})\}_{n_k}\xrightarrow{k\to \infty} x\ge 0,
\]
meaning that there exists $K>0$ for which
\[
x-\ve < \{\sigma_{j(n_k)}(A_{n_k})\}_{n_k}<x+\ve \qquad \forall k>K.
\]
Since $\serie A \sim_\sigma g$ we know that
\[
\lim_{n\to\infty} \frac{1}{n} \sum_{i=1}^{n} F(\sigma_i(A_n)) = \frac{1}{|D|}\int_D F(|g(x)|) dx
\]
for every function $F:\f R\to \f C$ continuous with compact support. Given such an $F$ real valued and with  $\chi_{[-\ve,x+2\ve]}\ge F\ge \chi_{[0,x+\ve]}$, we obtain
\[
\lim_{n\to\infty} \frac{1}{n} \sum_{i=1}^{n} F(\sigma_i(A_n)) \ge \limsup_{n\to\infty} \frac{1}{n} \left|\left\{ i:\sigma_i(A_n) \le x +\ve  \right\}\right|
\]and
\[
\int_D F(|g(x)|) dx \le |\{x:|g(x)|\le x+2\ve\}|.
\]
As consequence
\[
 \frac{|\{x:|g(x)|\le x+2\ve\}| }{|D|}\ge\limsup_{n\to\infty} \frac{1}{n} \left|\left\{ i:\sigma_i(A_n) \le x+\ve  \right\}\right|,
\]
\[
 \frac{|\{x:|g(x)|> x+2\ve \}| }{|D|}\le\liminf_{n\to\infty} \frac{1}{n} \left|\left\{ i:\sigma_i(A_n) > x +\ve \right\}\right|
\]
and we can find $M>0$ such that
\[
 \frac{|\{x:|g(x)|> x+2\ve \}| }{|D|} - \ve \le \frac{1}{n} \left|\left\{ i:\sigma_i(A_n) > x+\ve   \right\}\right| \qquad \forall n>M,
\]
\[
 \frac{|\{x:|g(x)|> x+2\ve \}| }{|D|} +x+2\ve   \le \frac{1}{n} \left|\left\{ i:\sigma_i(A_n) > x +\ve \right\}\right| + x + 3\ve \qquad \forall n>M.
\]
Given now an index $n_k$ such that $k>K$ and $n_k>\max\{N,M\}$, we know that
\[
\frac{1}{n_k} \left|\left\{ i:\sigma_i(A_{n_k}) > x +\ve \right\}\right| + x-\ve \]\[\le \frac{1}{n_k} \left|\left\{ i:\sigma_i(A_{n_k}) > \sigma_{j(n_k)}(A_{n_k}) \right\}\right| + \sigma_{j(n_k)}(A_{n_k})
\]
\[
= \frac{j(n_k)-1}{n_k} + \sigma_{j(n_k)}(A_{n_k}) 	\le L+\ve.
\]
Therefore we can rewrite
\[
 \frac{|\{x:|g(x)|> x+2\ve \}| }{|D|} +x+2\ve   \le L+\ve+ 4\ve 
\]
and this does not depend on $n$ anymore. This leads to the thesis since, if we choose
\[
F = \{x:|g(x)|\le x+2\ve \},
\]
then
\[
 p_m(g)= \inf_{E\cu D} \left\{ \frac{|E^C|}{|D|} + \ess\sup_E |g|  \right\}\le 
  \frac{|F^C|}{|D|} + \ess\sup_F |g|\le 
\]
\[
 \frac{|\{x:|g(x)|> x+2\ve \}| }{|D|} +x+2\ve \le L + 5\ve
\]
for all $\ve>0$, and eventually $p_m(g)\le L$.
\end{proof}

These lemmas lead to the desired result.

\begin{teo}\label{dis}
Given $\serie A\in\mc E$ and $g\in \mc M_D$, then
\[
\serie{A}\sim_\sigma g \implies   \rho(\serie{A}) = p_m(g) .
\]
\end{teo}

With the above results, we can finally focus only on the GLT algebra.

\section{The GLT Algebra}

Let us denote as $\mc C_D$ the set of the couples $(\serie A,k)\in \mc E\times \mc M_D$ such that $\serie A\sim_\sigma k$ and where $D=[0,1]\times[-\pi,\pi]$. First of all we can see that it is well defined, because from the definition, if $k,k'$ are two measurable functions that coincide almost everywhere, then 
\[
\serie A\sim_\sigma k\iff \serie A\sim_\sigma k'
\] 
so when we say that $k\in \mc M_D$ and $\serie A \sim_\sigma k$, it means that every function in the equivalence class $k$ is a spectral function for $\serie A$.

As already anticipated in the first section, the set $\mc G$ of GLT sequences is a subset of $\mc C_D$, so when we say that a sequence $\serie A$ is a GLT with spectral symbol $k$ and we write $\serie A\GLT k$, it means that $(\serie A,k)\in \mc G$ and in particular, it means that $\serie A\sim_\sigma k$. 

The GLT set is built so that for every $\serie A\in \mc E$ there exists at most one (class of) function $k$ such that $\serie A\GLT k$, so if we call $\mc H$ the GLT matrix sequences, that is the projection of $\mc G$ on $\mc E$, then there exists a map
\[
S : \mc H\to \mc M_D 
\]
that associates to each sequence its GLT spectral symbol
\[
S(\serie A) = k \iff \serie A\GLT k.
\]
The main properties of $\mc G$ that we need for the following sections and can be found or immediately derived from the results in \cite{GS} are
\begin{enumerate}
\item $\mc G$ is a $\f C$-algebra, meaning that given $(\serie A,k)$,$(\serie B,h)\in \mc G$ and \\$\lambda \in\f C$, then
\begin{itemize}
\item $(\{ A_n+B_n \}_n,k+h)\in \mc G$,
\item $(\{ A_nB_n \}_n,kh)\in \mc G$,
\item $(\{ \lambda A_n \}_n,\lambda k)\in \mc G$.
\end{itemize}
\item\label{app}
 Given $k\in \mc M_D$, then there exists $k_m\in\mc M_D$ that converge to $k$ in measure and such that they are GLT symbols.
\item \label{close}
$\mc G$ is closed in $\mc E\times \mc M_D$: given $\{(\B,k_m)\}_m\cu \mc G$ such that 
\[\B\acs\serie A, \qquad k_m\to k \text{ in measure}\]
 where $(\serie A,k)\in \mc E\times \mc M_D$, then $(\serie A,k)\in \mc G$.
\item \label{zero}We call \textit{zero-distributed} the matrix sequences that have 0 as spectral symbol, and we denote the sets 
\[
\mc Z = \{ (\serie C,0)\in \mc C_D \},\qquad \mc Z_M = \{ \serie C : (\serie C,0)\in \mc C_D \}.
\]
Then $\mc Z$ is a subalgebra of $\mc G$.
\item\label{ker} $S$ is an homomorphism of $\f C$-algebras and $\mc Z_M$ coincides with its kernel.
\end{enumerate}

With these powerful properties, we can prove that the GLT algebra is complete with respect to the metric on $\mc E\times \mc M_D$, and that $\mc H$, up to equivalence a.c.s., is actually isomorphic and isometric to $\mc M_D$.

\subsection{Cauchy Sequences}
An immediate result we can obtain from the precedent section is

\begin{corollario}\label{Cau}
Given $\{B_{n,m}\}_{n,m}\GLT f_m$, then
\[
\{B_{n,m}\}_{n,m} \text{ converges} \iff \{f_m\}_m \text { converges.}
\] 
\end{corollario}
\begin{proof}
Since both $\mc E$ and $\mc M_D$ are complete spaces with the respective metric, it is sufficient to prove that
\[
\{B_{n,m}\}_{n,m} \text{ Cauchy} \iff \{f_m\}_m \text { Cauchy}.
\]
We know that the space of GLT couples is a $\f C$-algebra, so
\[
\{B_{n,i}-B_{n,j}\}_{n}\GLT f_i - f_j.
\]
The previous section results imply that 
\[
\dacs{\{B_{n,i}\}_n}{\{B_{n,j}\}_n} = d_m(f_i,f_j)
\]
so if one of the two is a Cauchy sequence with respect to its metric, even the other one is a Cauchy sequence.
\end{proof}

 This is all we need to prove the following result.

\begin{teo}\label{mis}
Every measurable function $k\in \mc M_D$ with $D=[0,1]\times [-\pi,\pi]$ is a GLT symbol.
\end{teo}
\begin{proof}

Given $k$, by Property (\ref{app}) we know that there exists a sequence $k_m$ of GLT symbols that converge to $k$ in measure. If $\B\GLT k_m$, then by Corollary \ref{Cau}, $\B$ converges, so there exists $\serie A$ such that $\{B_{n,m}\}_{n,m}\acs \serie A$. Thanks to Property (\ref{close}), we obtain that $\serie A\GLT k$, so $k$ is a GLT symbol.
\end{proof}

This also lets us formulate other results that are easily provable.

\begin{corollario}
If the following conditions
\begin{enumerate}
\item $\B\GLT k_m$,
\item $\serie A\GLT k$,
\item $k_m\to k$ in measure,
\item $\B\acs \serie A$,
\end{enumerate}
are satisfied, then
\begin{itemize}
\item $(1),(3) \implies \exists \serie A : (4), (2)$,
\item $(1),(4) \implies \exists k: (2), (3)$,
\item $(2),(3) \implies \exists \B : (1),(4)$.
\end{itemize}
\end{corollario} 

In particular, a consequence is that $\mc H$  is a complete subspace of $\mc E$, since it is closed thanks to the statement $(1),(4) \implies \exists k: (2), (3)$ contained in the previous corollary.

\subsection{Isometry}

Let us define the a.c.s. equivalence on $\mc E$ as
\[
\serie A\sim_{acs} \serie B \iff \dacs{\serie A}{\serie B} = 0.
\]
A property of the zero distributed sequences is that
\[
\rho(\serie C) = 0\iff  \serie C\sim_\sigma 0.
\]
Consequently it is immediate to see that
\[ 
\serie A\sim_{acs} \serie B\iff \{A_n-B_n\}_n\sim_\sigma 0 \iff (\{A_n-B_n\}_n,0)\in\mc Z
\]
and thanks to the triangular inequality of $d_{acs}$, $\sim_{acs}$ is an equivalence relation in $\mc E$, and on all its subsets. We know that the set $\mc Z_M$ is a subalgebra of $\mc E$, and the previous lines show that
\[
\mc E / \sim_{acs} \equiv \mc E / \mc Z_M.
\]
As reported in Property \ref{ker}, $\mc Z_M$ is also the kernel of $S$, so we can define a new induced injective homomorphism of $\f C$-algebras
\[
T : \mc M := \mc H/\mc Z_M \hookrightarrow \mc M_D.
\]
The quotient preserves the distance, as shown in the following lemma.

\begin{lemma}
given $[A],[B]\in \mc M$, and given
\[\serie A,\serie {A'} \in [A],\qquad \serie B,\serie {B'}\in [B],\]
we have 
\[
\dacs{\serie A}{\serie B} = \dacs{\serie {A'}}{\serie {B'}}
\]
\end{lemma} 
\begin{proof}
The following relations hold
\[\serie A = \serie {A'} +\serie {Z}\qquad  \serie B = \serie {B'} +\serie {W}\qquad \serie {Z},\serie {W}\in \mc Z_M\]
 so that
\[
\dacs{\serie A}{\serie B} = \dacs{\serie {A'} +\serie {Z}}{\serie {B'} +\serie {W}}\]
\[ = \rho(\serie {A'} -\serie {B'} +\serie {Z} -\serie {W})
\]
\[
= \dacs{\serie {A'}-\serie {B'}}{\serie {Z}-\serie {W}}.
\]
However both $\serie {Z}-\serie {W}$ and $-\serie {Z}+\serie {W}$ are zero distributed, since $\mc Z_M$ is an algebra, and by Theorem \ref{dis}, we deduce
\[
\dacs{\serie {Z}-\serie {W}}{\serie 0} = \rho(\serie {Z}-\serie {W}) = p_m(0) = 0,
\]
\[
\dacs{-\serie {Z}+\serie {W}}{\serie 0} = \rho(\serie {Z}-\serie {W}) = p_m(0) = 0.
\]
By applying the triangular inequality, we obtain
\[
\dacs{\serie {A'}-\serie {B'}}{\serie {Z}-\serie {W}} \]\[\le \dacs{\serie {A'}-\serie {B'}}{\serie 0} + \dacs{\serie 0}{\serie {Z}-\serie {W}}\]\[ = \dacs{\serie {A'}-\serie {B'}}{\serie 0} = \rho(\serie {A'}-\serie {B'}) = \dacs{\serie {A'}}{\serie {B'}}
\]
and 
\[
\dacs{\serie {A'}-\serie {B'}}{\serie {Z}-\serie {W}} \]\[\ge \dacs{\serie {A'}-\serie {B'}}{\serie 0} - \dacs{\serie 0}{-\serie {Z}+\serie {W}}\]\[ = \dacs{\serie {A'}-\serie {B'}}{\serie 0} = \rho(\serie {A'}-\serie {B'}) = \dacs{\serie {A'}}{\serie {B'}}.
\]
As a consequence
\[
\dacs{\serie A}{\serie B} \]\[= \dacs{\serie {A'}-\serie {B'}}{\serie {Z}-\serie {W}} \]\[=\dacs{\serie {A'}}{\serie {B'}}.
\]
\end{proof}

Te latter result implies that $\mc M$ has still a pseudometric, defined as
\[
\dacs{[A]}{[B]}:= \dacs{\serie A}{\serie B}
\]
for any $\serie A\in[A]$ and $\serie B\in [B]$. The induced $d_{acs}$ is actually a real distance, since 
\[
\serie A\in [A],\quad \serie B\in [B],\quad \dacs{\serie A}{\serie B}= 0
\]
\[
\implies \serie A\sim_{acs}\serie B\implies [A]\equiv [B],
\]
so $\mc M$ is endowed with a complete metric. Moreover, since we operated a quotient of algebras, it is immediate to see that we can define a symbol GLT for every element of $\mc M$ as
\[
[A]\GLT k \iff \forall\serie A\in[A], \serie A\GLT k.
\]
We can eventually state and prove the main result of this section.

\begin{teo}
The map $T:\mc M\to \mc M_D$ is an isomorphism of $\f C$-algebras and an isometry of metric spaces
\end{teo}
\begin{proof}
$T$ is already an injective homomorphism of $\f C$-algebras, but in view of Theorem \ref{mis}, it is also surjective, so it is an isomorphism. If $[A],[B]\in \mc M$ with $[A]\GLT k$ and $[B]\GLT h$, and $\serie A\in [A]$, $\serie B\in [B]$, then
\[
\dacs{[A]}{[B]} = \dacs{\serie A}{\serie B} =\rho(\serie A-\serie B) 
\]
but $\serie A\GLT k$, $\serie B\GLT h$, and $\mc G$ is an algebra, so $\serie A-\serie B\GLT k-h$. Thanks to Theorem \ref{dis} we know that
\[
\rho(\serie A-\serie B) = p_m(k-h) = d_m(k,h) = d_m(T[A], T[B])
\]
hence
\[
\dacs{[A]}{[B]} = d_m(T[A], T[B])
\]
from which we conclude that $T$ is an isometry of metric spaces.
\end{proof}

\subsection{Maximality}

In this last section, we investigate the maximality of the GLT algebra in $\mc C_D$. Formally, let $\mc S_D$ be the set of the groups in $\mc C_D$. A group $G$ is characterized by the properties
\begin{itemize}
\item $a,b\in G\implies a+b \in G$,
\item $a\in G \implies -a\in G$,
\end{itemize}
so this means that $\mc G\in \mc S_D$.

 We can notice that Corollary \ref{Cau} can be generalized to any group $G\in\mc S_D$ with the same proof. 
 \begin{lemma}
Let $G\in \mc S_D$, and $(\{B_{n,m}\}_{n,m},f_m)\cu G$. Then
 \[
 \{B_{n,m}\}_{n,m} \text{ converges} \iff f_m \text { converges.}
 \] 
 \end{lemma}
 
 We have already seen that $\mc E$ and $\mc M_D$ are complete spaces, along with their product. We proved that even the set $\mc G$ is a complete space, and it is easy to reprove it in a different way: $\mc G$ is closed in $\mc E \times \mc M_D$, and a closed set in a complete space is complete. The same reasoning can be extended to groups.
 \begin{lemma}
Let $G\in \mc S_D$. Then $M$ is complete if and only if it is closed in $\mc E \times \mc M_D$.
 \end{lemma}

Even Theorem \ref{mis} is generalizable, but we need additional assumptions.
  \begin{teo}
 Let $G\in \mc S_D$ be a group containing $\mc Z$. Then every measurable function $k$ is a spectral symbol for $G$ if and only if 
 \begin{itemize}
 \item The set of spectral symbol for $G$ is dense in $\mc M_D$.
 \item $G$ is closed in $\mc E \times \mc M_D$.
 \end{itemize}
 and in this case, $G$ is a maximal element in $\mc S_D$ if we impose the inclusion partial order on it.
  \end{teo}
  \begin{proof}
If the two conditions are true, then the proof is analogous to the one of Theorem \ref{mis}. 
 
For the converse, if all measurable functions are spectral symbols for $G$, then the first condition is surely verified. Let
	\[(\B,k_m)\cu G,\qquad (\B,k_m)\to (\serie A,k).\]
We know that there exists $(\serie C,k)\in G$, so 
\[ (\{C_n-B_{n,m}\},k-k_m) \in M, \qquad k-k_m\to 0\]
and 
\[
\B\acs \serie C, \qquad \B\acs \serie A
\]
\[ \dacs{\serie A}{\serie C}\le \dacs{\serie A}{\{B_{n,m}\}_n} +  \dacs{\serie C}{\{B_{n,m}\}_n} \to 0\]
\[\implies \dacs{\serie A}{\serie C} = 0 \implies \serie A-\serie C\in \mc Z_M.\]
We know that $\mc Z\cu G$, so we conclude
\[
(\serie A,k) = (\serie C,k) + (\serie A-\serie C,0)\in G.
\]

To show that $G$ is a maximal element in $\mc S_D$, let $N$ be a group in $\mc S_D$ containing $G$, and let us consider any couple $(\serie{A},k)\in N$. We know that there exists $(\serie C,k)\in G\cu N$, but $N$ is a group, so $(\{A_n-C_n\}_n,0)\in N$. This is also a zero distributed sequence, so it belongs to $G$, and, as before, 
\[(\serie A,k) = (\serie C,k) + (\serie A-\serie C,0)\in G.\]
This concludes that $N=G$, so $G$ is maximal in $\mc S_D$.
  \end{proof}

If $D = [0,1]\times [-\pi,\pi]$, then we can apply the previous theorem, for obtaining the following result.

\begin{corollario}
$\mc G$ is a maximal group in $\mc{S}_D$.
\end{corollario}
%
%

The last question left to see is if $\mc G$ is also a maximum element in $\mc S_D$, meaning that it contains all the other groups. The answer, in this case, is negative, since there exists a lot of trivial transformations of the variables that do not change the spectral symbols properties. For example
\[
GLT^{inv} = \Set{(\serie{A},k(1-x,\theta)) | \serie{A}\GLT k(x,\theta)}
\]
is a closed algebra in $\mc C_D$ since \[\serie{A}\GLT k(x,\theta)\implies \serie{A}\sim_\sigma k(x,\theta) \implies \serie{A}\sim_\sigma k(1-x,\theta).\]
In general, any symmetry or cyclic translation of the variables produces spectral symbols for the same sequences, and they respect all the axiom of closed algebra since the operations of sum/product/convergence commute with the transformations. Eventually, GLT$^{inv}$ and GLT are incompatible, since there exists $k$ that is not symmetric on the variable $x$, so 
\[
\serie{A}\GLT k(x,\theta), 	\quad   k(x,\theta)\ne k(1-x,\theta)\]
\[ \implies \serie{A}\not\sim_{GLT} k(1-x,\theta),\quad \serie{A}\not\sim_{GLT^{inv}} k(x,\theta)
\]
and this means that they are not contained in each other.

\bibliography{ArtGLT}

\bibliographystyle{plain}

\end{document}